\pdfoutput=1
\documentclass[11pt,leqno]{article}

\usepackage{bm}
\usepackage{amsmath,amsthm,amssymb,amsfonts}
\usepackage{graphicx,subfigure,booktabs,enumerate,url,authblk} 
\usepackage{mathabx}
\usepackage{algorithm,algpseudocode}
\usepackage{color}
\usepackage{listings}
\usepackage{bussproofs}
\usepackage{latexsym}

\def\Ebb{\mathbb{E}}
\def\Rbb{\mathbb{R}}

\theoremstyle{plain}
\newtheorem{thm}{Theorem}[section]

\theoremstyle{remark}
\newtheorem{rem}[thm]{Remark}

\theoremstyle{definition}

\makeatletter

\numberwithin{equation}{section}
\numberwithin{figure}{section}
\numberwithin{table}{section}

\algnewcommand\algorithmicto{\textbf{to}}
\algdef{SE}[FOR]{ForTo}{EndFor}[2]{%
  \algorithmicfor\ #1\ \algorithmicto\ #2\ \algorithmicdo
}{%
  \algorithmicend\ \algorithmicfor
}

\title{Weak approximation of Schr{\"o}dinger-F{\"o}llmer diffusion}
\author{Koya Endo\thanks{E-mail: kouya.endou@gmail.com}}
\author[1]{Yumiharu Nakano\thanks{Corresponding author. E-mail: nakano@c.titech.ac.jp}}
\affil[1]{Department of Mathematical and Computing Science, School of Computing, 
 	    Tokyo Institute of Technology}

\date{\today}

\begin{document}

\maketitle

\begin{abstract}
We consider the weak convergence of the Euler-Maruyama approximation for Schr{\"o}dinger-F{\"o}llmer diffusions, which are solutions of 
Schr{\"o}dinger bridge problems and can be used for sampling from given distributions. 
We show that the distribution of the terminal random variable of the time-discretized process weakly converges to the target one under mild regularity conditions. 
\begin{flushleft}
\end{flushleft}
\begin{flushleft}
\end{flushleft}
\end{abstract}

%49N45 Inverse problems in optimal control
%93E12 Identification in stochastic control theory
%35K10 Second order parabolic equations
%35K15 Initial value problems for second order parabolic equations
%35K55 Nonlinear parabolic equations
%35D40 Viscosity solutions
%65M70 Spectral, collocation and related methods
%60H35 Computational methods for stochastic equations (aspects of stochastic analysis)
%62D05 Sampling theory, sample surveys

%%%

\section{Introduction}\label{sec:1}

In this paper we are concerned with a weak solution of the following stochastic differential equation: 
\begin{equation}
\label{SFD}
dX_t = \nabla \log h(t,X_t) dt + dB_t, \quad 0\le t\le T, \;\;  X_0 = 0. 
\end{equation}
Here, $\nabla$ denotes the gradient operator, $\{B_t\}_{0\le t\le T}$ is a $d$-dimensional standard Brownian motion on some probability space, and 
for a given Borel probability measure $\mu$ on $\mathbb{R}^d$ with positive density $\rho$ the function $h$ is defined by 
\begin{equation}
\label{def:h}
 h(t,x) := \int_{\mathbb{R}^d}\varphi(y)p_{T-t}(x,y)dy, \quad (t,x)\in [0,T]\times\mathbb{R}^d, 
\end{equation}
with 
\begin{equation}
\label{def:varphi}
 \varphi(y) = \frac{\rho(y)}{p_T(0,y)}, \quad y\in\mathbb{R}^d, 
\end{equation}
where $p_s$ is the transition density function of $\{B_t\}$. 

Eq.~\eqref{SFD} appears in the so-called {\it Schr{\"o}dinger bridge problem} between the Dirac measure $\delta_0$ with mass at origin and $\mu$. 
To explain the Schr{\"o}dinger bridge problem briefly, let $\mathbb{W}^d = C([0,T], \Rbb^d)$ be the space of all $\Rbb^d$-valued continuous functions on $[0,T]$. 
Denote by $\mathcal{P}(\mathbb{W}^d)$ the totality of Borel probability measures on $\mathbb{W}^d$. 
Assume here that \eqref{SFD} has a weak solution and denote by $P^*$ the corresponding path measure on $\mathbb{W}^d$. 
Further denote by $D_{KL}(Q\,\|\, P)$ the Kullback-Leibler divergence or the relative entropy of $Q\in\mathcal{P}(\mathbb{W}^d)$ with respect to 
the Wiener measure $P\in\mathcal{P}(\Rbb^d)$, 
defined by   
\[
 D_{KL}(Q\,\|\,P) = 
 \begin{cases}
   \Ebb_Q\left[\log\dfrac{dQ}{dP}\right],  & \text{if}\;\; Q \ll P, \\
  +\infty, &\text{otherwise}, 
 \end{cases}
\]
where $\mathbb{E}_R$ denotes the expectation with respect to a probability measure $R$ on a measurable space. 
If $D_{KL}(P^*\,\|\,P)<\infty$, then $P^*$ is a unique solution of 
the minimization problem of $D_{KL}(Q\,\|\,P)$ over all $Q\in\mathcal{P}(\mathbb{W}^d)$ such that 
the marginal distributions of $Q$ at time $0$ and $1$ are given by $\delta_0$ and $\mu$, respectively.

The name Schr{\"o}dinger bridge problem comes from Erwin Schr{\"o}dinger's works 
\cite{schr:1931} and \cite{schr:1932}. His aim was to study a transition probability that most likely occurs under 
constraints on the initial and terminal time distributions of the empirical measures of independent Brownian particles.  
We refer to Chetrite et al.~\cite{chetrite-etal:2021}, an english translation  of \cite{schr:1931}, 
for an exposition of Schr{\"o}dinger's original approach. 
F{\"o}llmer \cite{fol:1988} discovers that such problem is nothing but the one of large deviation  
and is nearly equivalent to the minimization problem above. Moreover,  \cite{fol:1988} shows the optimality of $P^*$. 
In particular, the marginal distribution $P^*$ at time $T$ coincides with $\mu$. 
We refer to, e.g., Chen et al.~\cite{chen-etal:2021} and L{\'e}onard \cite{leo:2013} for a detailed survey of Schr{\"o}dinger's bridges.

Sampling from a given probability distribution is an important issue, primarily in fields such as statistics and machine learning. 
For instance, in Bayesian estimation, the accuracy of sampling from an unnormalized posterior distribution is a critical concern. 
Thus, various sampling methods have been studied, including Markov chain Monte Carlo methods and the Langevin samplers.
Recently, Huang et al.~\cite{hua-etal:2021} proposes to use Schr{\"o}dinger bridges as a sampler of $\mu$. 
Based on this sampling method, Dai et al.~\cite{dai-etal:2023} applies Schr{\"o}dinger bridges to global optimization. 
Following \cite{hua-etal:2021} and \cite{dai-etal:2023}, we call a weak solution of \eqref{SFD} the {\it Schr{\"o}dinger-F{\"o}llmer diffusion}. 

In sampling applications, we need to consider a time discretization of the Schr{\"o}dinger-F{\"o}llmer diffusion, which are also discussed in 
\cite{hua-etal:2021} and \cite{dai-etal:2023}. The both studies assume that the drift function in \eqref{SFD} is of linear growth, Lipschitz continuous in 
$x$, and $1/2$-H{\"o}lder continuous in $t$ on $[0,T]\times\Rbb^d$, and then  
justify the Euler-Maruyama method for \eqref{SFD} as the {\it strong approximation}. 
In many practical situations, e.g., in Bayesian estimation,  
the density $\rho$ has a complicated form, and so in general it may be difficult to check that the drift function satisfies these conditions.  
In the present paper, we aim to find weaker sufficient conditions for which the Euler-Maruyama method for the Schr{\"o}dinger-F{\"o}llmer diffusion converges. 
In particular, we discuss the {\it weak approximation} of \eqref{SFD}, carried out in the next section.

\section{Main results}

Let $\mathcal{P}(\mathbb{R}^d)$ be the set of all Borel probability measures on $\mathbb{R}^d$.
Let $\mu\in\mathcal{P}(\Rbb^d)$ with density $\rho$ satisfying the following condition:
\begin{enumerate}
\item[(A1)] The function $\rho$ is positive on $\mathbb{R}^d$ such that 
\[
 \int_{\Rbb^d}\rho(y)(|y|^2 + \log\rho(y))dy < \infty. 
\]
\end{enumerate}
Under (A1) there exists a weak solution of \eqref{SFD} and the corresponding path measure $P^*$ satisfies 
\[
 D_{KL}(P^*\,\|\,P) = \Ebb_{P^*}\int_0^T|\nabla \log h(t,x_t)|^2 dt <\infty, 
\]
where $\{x_t\}$ is the coordinate map on $\mathbb{W}^d$ (see \cite{fol:1988} and Dai-Pra \cite{dai:1991}). 

Let us consider the Euler-Maruyama approximation of the weak solution of \eqref{SFD}. 
To this end, let $(\Omega,\mathcal{F},\mathbb{P})$ be an atomless and complete probability space, and 
$\{Z_i\}_{i=1}^n$ an IID sequence on $(\Omega,\mathcal{F},\mathbb{P})$ such that 
each $Z_i$ is a $d$-dimensional standard normal vector. 
Let $\{t_i\}_{i=0}^n$ be given by $t_i=iT_n$, $i=0,\ldots,n$. 
Define $\{Y_i\}_{i=0}^n$ by 
\begin{equation}
\label{sampler1}
Y_{i+1} = Y_i + b(t_i,Y_i) h + \sqrt{h} Z_{i+1}, \quad i=0, 1, \cdots, n-1, 
\end{equation}
with $Y_0=0$, where $h=T/n$ and  
\[
b(t, y) = \nabla \log h(t,y) = \frac{\nabla \Ebb[ \varphi(y + \sqrt{T-t} Z)]}{\Ebb[\varphi(y + \sqrt{T-t} Z)]}, \quad (t,y)\in [0,T)\times\Rbb^d, 
\]
where $\mathbb{E}=\mathbb{E}_{\mathbb{P}}$. 

In the framework of weak solutions, it is convenient to work under the boundedness of drift functions. 
To this end, we shall impose the following condition on $\mu$: 
\begin{enumerate}
\item[(A2)] The function $\log \varphi$ is Lipschitz continuous on $\mathbb{R}^d$. 
\end{enumerate}

Here is our first main result.  
\begin{thm}
\label{thm:1st}
Let $\{Y_i\}_{i=0}^n$ be as in $\eqref{sampler1}$. Under $(A1)$ and $(A2)$, 
the distribution of $Y_n$ under $\mathbb{P}$ weakly converges to $\mu$ as $n\to\infty$. 
\end{thm}
\begin{rem}
Suppose that $\rho(x)\,\propto\, e^{-V(x)}$. Then a sufficient condition for which (A2) holds is that 
\[
 \tilde{V}(x) := V(x) - \frac{1}{2T}|x|^2, \quad x\in\Rbb^d, 
\]
is globally Lipschitz.  
\end{rem}

\begin{rem}
Huang et al.~\cite{hua-etal:2021} assumes that the drift $b$ is of linear growth, Lipschitz continuous in $x$, and $1/2$-H{\"o}lder continuous in $t$ on 
$[0,T]\times\mathbb{R}^d$. Then they justifies a strong approximation of the Euler-Maruyama method. Remark 4.1 in \cite{hua-etal:2021} states that 
a set of sufficient conditions for which these hold is that $\varphi$ is bounded from below and both $\varphi$ and $\nabla\varphi$ are Lipschitz continuous. 
If $\rho(x)\,\propto\, e^{-V(x)}$ as in the previous remark, then this requires that $\tilde{V}$ has bounded derivatives up to second order, 
which is stronger than our conditions (A1) and (A2). 
\end{rem}

\begin{rem}
It should be noted that the weak convergence of the Euler-Maruyama methods under non-Lipschitz conditions for the drift term has already been studied, 
for example, in Zhang \cite{zha:2019}. To apply the results from \cite{zha:2019} to \eqref{SFD}, the drift function $\nabla \log h(t,x)$ in \eqref{SFD} must be continuous in $t$ 
on $[0,T]$ for any $x$. However, under our conditions (A1) and (A2), the continuity of $t \mapsto \nabla \log h(t,x)$ at $T$ cannot be assured. 
That is, Schr{\"o}dinger-F{\"o}llmer diffusions generally fall outside the scope of existing studies on the weak convergence of the Euler-Maruyama methods.
\end{rem}

\begin{proof}[Proof of Theorem $\ref{thm:1st}$]
Step (i). First we confirm that $b$ can be defined as an essentially bounded function on $[0,T]\times\Rbb^d$.
By (A2), we can apply Rademacher's theorem (see, e.g., Evans \cite{eva:1998}) to find that 
$\nabla\log\varphi$ exists almost everywhere in $\Rbb^d$ and $|\nabla\varphi(x)|=|\varphi(x)\nabla\log\varphi(x)|\le C_0\varphi(x)$ for almost every $x\in\Rbb^d$ 
for some constant $C_0>0$.
Thus, by the dominated convergence theorem
\[
 b(t,x) = \frac{\Ebb[\nabla\varphi(x + \sqrt{T-t}Z)]}{\Ebb[\varphi(x+\sqrt{T-t}Z)]}, \quad \text{a.e.}\, x\in \mathbb{R}^d, \;\; 0\le t\le T.   
\] 
From this representation we can naturally define $b(T,\cdot)$ by  
$b(T,x)=\nabla\varphi(x)/\varphi(x)$ a.e.~$x$. With this definition, we obtain 
$|b(t,x)|\le C_0$ a.e.~ $x\in\Rbb^d$ for any $t\in [0,T]$. 

Step (ii). Let $\{X_t\}_{0\le t\le T}$ be a $d$-dimensional standard Brownian motion on $(\Omega,\mathcal{F},\mathbb{P})$. 
By Step (i), the process $b(t,X_t)$, $0\le t\le T$, is bounded $\mathbb{P}$-a.s., whence 
we can define the probability measure $\mathbb{P}^*$ on $(\Omega,\mathcal{F})$ by 
\[
 \frac{d\mathbb{P}^*}{d\mathbb{P}} = \exp\left[\int_0^Tb(s,X_s)^{\mathsf{T}}dX_s - \frac{1}{2}\int_0^T|b(s,X_s)|^2ds\right], 
\]
where $a^{\mathsf{T}}$ denotes the transposition of a vector $a$. 
Further, by Girsanov's theorem, the process 
\[
 B_t^* := X_t - \int_0^t b(s,X_s)ds, \quad 0\le t\le T, 
\]
is a Brownian motion under $\mathbb{P}^*$. Thus, the $(\Omega,\mathcal{F},\mathbb{F},\mathbb{P}^*,X, B^*)$ is a weak solution of \eqref{SFD}, i.e., 
\begin{equation}
\label{eq:weak-sol1}
 X_t = \int_0^tb(s,X_s)ds + B^*_t. 
\end{equation}
In particular, $\mathbb{P}^*(X_T)^{-1}=\mu$. 

Step (iii). 
Define $\tilde{X}_t$ by 
\[
 \tilde{X}_t = \tilde{X}_{t_i} + b(t_i, \tilde{X}_{t_i})(t-t_i) + B^*_t - B^*_{t_i}, \quad t_i< t\le t_{i+1}, \;\; i=0, \ldots n-1, 
\]
with $\tilde{X}_0=0$. Note that $\mathbb{P}^*(\tilde{X}_T)^{-1}=\mathbb{P}(Y_n)^{-1}$. 
Further $\{\tilde{X}_t\}$ satisfies 
\begin{equation}
\label{eq:weak-sol1}
 \tilde{X}_t = \int_0^t b(\tau(s),\tilde{X}_{\tau(s)})ds + B_t^*, 
\end{equation}
where $\tau_n(s)=\lfloor ns\rfloor / n$ and $\lfloor x\rfloor$ denotes the greatest integer not exceeding $x\in\mathbb{R}$. 
In particular, $\{B_t^*\}$ is adapted to the augmented natural filtration $\mathbb{G}$ generated by $\{\tilde{X}_t\}$. 
Then put $\beta_t=b(\tau_n(t),\tilde{X}_{\tau_n(t)}) - b(t,\tilde{X}_t)$ and consider another probability measure 
$\hat{\mathbb{P}}$ on $(\Omega,\mathcal{F})$ defined by 
\[
 \frac{d\hat{\mathbb{P}}}{d\mathbb{P}^*} = \exp\left[-\int_0^T\beta_s^{\mathsf{T}}dW^*_s - \frac{1}{2}\int_0^T|\beta_s|^2ds\right].
\]
Again by Girsanov's theorem, the process
\[
 \hat{B}_t := B^*_t + \int_0^t\beta_sds, \quad 0\le t\le T, 
\]
is a $\hat{\mathbb{P}}$-Brownian motion. This leads to 
\begin{equation}
\label{eq:weak-sol3}
\tilde{X}_t = \int_0^tb(s,\tilde{X}_s)ds + \hat{B}_t, \quad 0\le t\le T, 
\end{equation}
whence $(\Omega,\mathcal{F},\mathbb{F},\hat{\mathbb{P}},\tilde{X}, \hat{B})$ is also a weak solution of \eqref{SFD}. 
By the uniqueness in law for the weak solution of \eqref{SFD} obtained by Girsanov's theorem, we have 
$\mathbb{P}^*X^{-1} = \hat{P}\tilde{X}^{-1}$ (see Proposition 5.3.10 in Karatzas and Shreve \cite{kar-shr:1991}). 

Now, for any $\Gamma\in \mathcal{B}(\mathbb{W}^d)$, 
\[
 \mathbb{P}^*(X\in\Gamma) = \mathbb{E}_{\hat{\mathbb{P}}}\left[\frac{d\mathbb{P}^*}{d\hat{\mathbb{P}}}1_{\{X\in\Gamma\}}\right] 
 =\mathbb{E}_{\hat{\mathbb{P}}}\left[\exp\left(\int_0^T\beta_s^{\mathsf{T}}d\hat{W}_s - \frac{1}{2}\int_0^T|\beta_s|^2ds\right)1_{\{X\in\Gamma\}}\right]. 
\]
Since $\beta$ is $\mathbb{G}$-adapted, as in the proof of Lemma 2.4 in \cite{kar-shr:1991}, 
we have $\int_0^T\beta_s^{\mathsf{T}}d\hat{B}_s = \lim_{k\to\infty}\int_0^T(\beta_s^{(k)})^{\mathsf{T}}d\hat{B}_s$ a.s. possibly along subsequence for some 
$\mathbb{G}$-adapted simple processes $\{\beta_t^{(k)}\}_{0\le t\le T}$, $k\in\mathbb{N}$. The process $\hat{B}$ is also $\mathbb{G}$-adapted, whence 
there exists a $\mathcal{B}(\mathbb{W}^d)$-measurable map $\Phi$ such that 
\[
 \Phi(\tilde{X})=\exp\left(\int_0^T\beta_s^{\mathsf{T}}d\hat{B}_s - \frac{1}{2}\int_0^1|\beta_s|^2ds\right), \quad\hat{\mathbb{P}}\text{-a.s.}
\]
By exactly the same way, we see 
\[
 \Phi(X) = \exp\left(\int_0^T(\theta^{(n)}_s)^{\mathsf{T}}d B^*_s - \frac{1}{2}\int_0^T|\theta^{(n)}_s|^2ds\right), \quad\mathbb{P}^*\text{-a.s.}, 
\]
where $\theta_s^{(n)} = b(\tau_n(s),X_{\tau_n(s)}) - b(s,X_s)$, $0\le s\le T$. 
This means 
\[
\mathbb{P}^*(X\in\Gamma) = \mathbb{E}_{\mathbb{P}^*}[Z_T1_{\{X\in\Gamma\}}]
\]
where 
\[
 Z_t = \exp\left(\int_0^t(\theta^{(n)}_s)^{\mathsf{T}}d B^*_s - \frac{1}{2}\int_0^t|\theta^{(n)}_s|^2ds\right), \quad 0\le t\le T. 
\]

Step (iii). Let $A\in\mathcal{B}(\Rbb^d)$. Using Birkholder-Davis-Gundy inequality, Cauchy-Schwartz inequality, 
Doob's maximal inequality, and the boundedness of $b$, we observe 
\begin{align*}
\left|\mathbb{P}^*(X_T\in A) - \mathbb{P}^*(\tilde{X}_T\in A)\right| 
&\le \mathbb{E}_{\mathbb{P}^*}|Z_T-1|\le C\mathbb{E}_{\mathbb{P}^*}\left[\left(\int_0^TZ_s^2|\theta^{(n)}_s|^2ds\right)^{1/2}\right]  \\ 
&\le C\mathbb{E}_{\mathbb{P}^*}\left[\sup_{0\le s\le T}Z_s\left(\int_0^T|\theta^{(n)}_s|^2ds\right)^{1/2}\right] \\ 
&\le C\mathbb{E}_{\mathbb{P}^*}\left[\sup_{0\le s\le T}Z_s^2\right]^{1/2}\mathbb{E}_{\mathbb{P}^*}\left[\int_0^T|\theta^{(n)}_s|^2ds\right]^{1/2} \\ 
&\le C\mathbb{E}_{\mathbb{P}^*}\left[\int_0^T|\theta^{(n)}_s|^2ds\right]^{1/2}. 
\end{align*}
for some positive constants $C$, which may be different from line to line. 
Hence
\[
 \left|\mathbb{P}(Y_n\in A) - \mathbb{P}^*(\tilde{X}_T\in A)\right|\le C\mathbb{E}_{\mathbb{P}^*}\left[\int_0^T|\theta^{(n)}_s|^2ds\right]^{1/2}. 
\]
Now, by the continuity of $b$ on $[0,T)\times\mathbb{R}^d$, 
\[
 \lim_{n\to\infty}b(\tau_n(s),X_{\tau_n(s)}) = b(s, X_s), \quad dt\times d\mathbb{P}^*\text{-a.e.}. 
\]
Thus, applying the bounded convergence theorem, we obtain 
\[
\lim_{n\to\infty}\mathbb{E}_{\mathbb{P}^*}\left[\int_0^T|\theta^{(n)}_s|^2ds\right]=0. 
\]
Consequently, we have shown that the total variation distance between $\mathbb{P}(Y_n)^{-1}$ and $\mu$ converges to zero 
as $n\to\infty$. 
\end{proof}

Next, we shall consider the case where the density $\rho$ may have a compact support. Again we impose the following integrability condition as in (A1): 
\begin{enumerate}
\item[(A3)] With usual convention $0\cdot \log 0 =0$, we have 
\[
 \int_{\Rbb^d}\rho(y)(|y|^2 + \log\rho(y))dy < \infty. 
\]
\end{enumerate}
It should be emphasized here that $\rho$ is not necessarily a positive function, whence \eqref{SFD} may not have a weak solution. 
Thus we introduce the approximate distribution $\mu_{\varepsilon}$ with density $\rho_{\varepsilon}$ defined by 
\[
 \rho_{\varepsilon}(y) = (1-\varepsilon)\rho(y) + \varepsilon G_T(y), \quad y\in\mathbb{R}^d, 
\]
where $\varepsilon\in (0,1)$ and 
\[
 G_T(y) = p_T(0,y) = \frac{e^{-|y|^2/(2T)}}{(2\pi T)^{d/2}}, \quad y\in\mathbb{R}^d. 
\]

Since $\lim_{r\searrow 0}r\log r=0$, the function $F(\xi):=\xi\log\xi -\xi + 1$ is nonnegative, continuous, and convex on $[0,\infty)$. Thus, 
\[
 0\le \int_{\Rbb^d}\rho_{\varepsilon}(y)\log\rho_{\varepsilon}(y)dy \le (1-\varepsilon) \int_{\Rbb^d}\rho(y)\log\rho(y) dy 
  + \varepsilon \int_{\Rbb^d}G_T(y)\log G_T(y)dy < \infty.  
\] 
This means that $\mu_{\varepsilon}$ satisfy (A1), and so there exists a weak solution of 
\begin{equation}
\label{SFD:eps}
 dX_t = \nabla\log h_{\varepsilon}(t,X_t)dt + dB_t, \quad X_0=0. 
\end{equation}
Here, $h_{\varepsilon}$ is defined as in \eqref{def:h} with $\rho$ replaced by $\rho_{\varepsilon}$, i.e., 
\[
 h_{\varepsilon}(t,x) = \int_{\Rbb^d}\varphi_{\varepsilon}(x + \sqrt{T-t}z)G_1(z)dz, \quad (t,x)\in [0,T]\times\mathbb{R}^d
\] 
where
\[
 \varphi_{\varepsilon}(y) = \frac{\rho_{\varepsilon}(y)}{G_T(y)} = (1-\varepsilon)\varphi(y) + \varepsilon, \quad y\in\Rbb^d, 
\]
and $\varphi$ is as in \eqref{def:varphi}, i.e., 
\[
 \varphi(y) = \frac{\rho(y)}{G_T(y)}, \quad y\in\Rbb^d. 
\]

Our plan is first to construct a weak approximation of \eqref{SFD:eps} and then letting $\varepsilon\to 0$. 
Let $\{Z_i\}_{i=1}^n$ and $\{t_i\}_{i=0}^n$ be as above. 
Define $\{Y_i\}_{i=0}^n$ by 
\begin{equation}
\label{sampler2}
Y_{i+1} = Y_i + b_{\varepsilon} (t_i,Y_i) h + \sqrt{h} Z_{i+1}, \quad i=0, 1, \cdots, n-1, 
\end{equation}
with $Y_0=0$, where $h=T/n$ and for $(t,y)\in [0,T)\times\Rbb^d$, 
\[
b_{\varepsilon}(t, y) = \nabla \log h_{\varepsilon}(t,y) = \frac{\nabla \Ebb[ \varphi_{\varepsilon}(y + \sqrt{T-t} Z)]}{\Ebb[\varphi_{\varepsilon}(y + \sqrt{T-t} Z)]} 
                                = \frac{(1-\varepsilon)\nabla\Ebb[\varphi(y + \sqrt{T-t} Z)]}{(1-\varepsilon)\Ebb[\varphi(y + \sqrt{T-t} Z)] + \varepsilon}.  
\]

To guarantee the boundedness of $b_{\varepsilon}$ on $[0,T]\times\mathbb{R}^d$ except null sets, we impose the following condition on $\varphi$: 
\begin{enumerate}
\item[(A4)] There exists a positive constant $C_1$ such that 
\[
 |\varphi(x)-\varphi(y)|\le C_1(1 + \varphi(x) + \varphi(y))|x-y|, \quad x,y\in\Rbb^d. 
\]
\end{enumerate}

As a criterion of convergences, we adopt the total variation distance $\delta$ defined by 
\[
 \delta(\nu_1,\nu_2) = \sup_{A\in\mathcal{B}(\Rbb^d)}\left|\nu_1(A) - \nu_2(A)\right|, \quad \nu_1,\nu_2\in\mathcal{P}(\Rbb^d). 
\]
Then we have the following: 
\begin{thm}
\label{thm:2nd}
Let $\{Y^{\varepsilon}_i\}_{i=0}^n$ be as in $\eqref{sampler2}$. 
Denote by $\mathcal{L}(Y_n^{\varepsilon})$ the distribution of $Y_n^{\varepsilon}$ under $\mathbb{P}$. 
Under $(A3)$ and $(A4)$, we have 
\[ 
 \lim_{\varepsilon\to 0}\lim_{n\to\infty}\delta(\mathcal{L}(Y_n^{\varepsilon}), \mu)=0.  
\] 
\end{thm}

\begin{rem}
Huang et al.~\cite{hua-etal:2021} analyses the weak convergence of $\mathcal{L}(Y_n^{\varepsilon})$ under the condition that 
$b_{\varepsilon}$ is of linear growth, Lipschitz continuous in $x$, and $1/2$-H{\"o}lder continuous in $t$ on $[0,T]\times\mathbb{R}^d$ for any $\varepsilon>0$. 
A set of sufficient conditions for which these hold is 
the Lipschitz continuity of both $\varphi$ and $\nabla\varphi$. See Section 4.3 in \cite{hua-etal:2021}. 
If $\varphi$ is Lipschitz, then $\rho(x)\le C(1+|x|)G_T(x)$ for some constant $C>0$ from which (A3) holds. In this sense, our conditions (A3) and (A4) are 
weaker than those imposed in \cite{hua-etal:2021}.  
\end{rem}

\begin{rem}
Consider the case where $\rho$ is given by the density estimator with triangular kernel, i.e., 
\[
 \rho(x) = \frac{1}{mh_0}\sum_{j=1}^m\left(1 - \left|\frac{x-x_j}{h_0}\right|\right)_+, \quad x\in\mathbb{R} 
\]
for some $x_1,\ldots,x_m\in\mathbb{R}$, $h_0>0$ and $(x)_+=\max(x,0)$ for $x\in\mathbb{R}$. 
Then it is straightforward to see that $\varphi_j(x):=(1-|x-x_j|/h_0)_+e^{|x|^2/(2T)}$ satisfies (A4) and that $\varphi^{\prime}_j$ is not continuous on $\mathbb{R}$. 
Therefore, in this case, the conditions (A3) and (A4) are satisfied, but not the Lipschitz continuity of $\varphi^{\prime}$. 
\end{rem}

\begin{proof}[Proof of Theorem $\ref{thm:2nd}$]
Let us confirm that $b_{\varepsilon}$ can be defined as an essentially bounded function on $[0,T]\times\Rbb^d$. 
By (A4), the function $\varphi$ is locally Lipschitz. This together with Rademacher's theorem tells us that 
$\nabla\varphi$ exists almost everywhere in $\Rbb^d$. Then again by (A4), 
$|\nabla\varphi(x)|\le C_1(1+2\varphi(x))$ for almost every $x\in\Rbb^d$. Thus, applying the dominated convergence theorem, we find 
\[
 b_{\varepsilon}(t, x) = \frac{(1-\varepsilon)\Ebb[\nabla\varphi(x + \sqrt{T-t} Z)]}{(1-\varepsilon)\Ebb[\varphi(x + \sqrt{T-t} Z)] + \varepsilon}, 
 \quad \text{a.e.}\, x\in\Rbb^d, \;\; 0\le t<T.  
\]
From this representation we can naturally define $b_{\varepsilon}(T,\cdot)$ by  
\[
 b_{\varepsilon}(T,x) = \frac{(1-\varepsilon)\nabla\varphi(x)}{(1-\varepsilon)\varphi(x) + \varepsilon}, \quad \text{a.e.}\,x\in\Rbb^d.  
\]
Since $|\nabla\varphi|\le C_1(1+2\varphi)$, we obtain 
\[
 |b_{\varepsilon}(t,x)|\le \frac{C_1(1-\varepsilon)}{\varepsilon} + 2C_1, \quad\text{a.e.}\, x\in\Rbb^d, 
\]
for any $t\in [0,T]$. 
Therefore, the process $b_{\varepsilon}(t,X_t)$, $0\le t\le T$, is bounded $\mathbb{P}$-a.s. 
Thus we can proceed exactly the same way as in the proof of Theorem \ref{thm:1st}, and obtain 
\[
 \delta(\mathcal{L}(Y_n^{\varepsilon}),\mu_{\varepsilon})\le C_{\varepsilon}\mathbb{E}_{\mathbb{P}^*}\left[\int_0^T|\theta^{(n)}_s|^2ds\right]^{1/2}, 
\]
where $\mathbb{P}^*$ and $\theta^{(n)}$ are defined as in the proof of Theorem \ref{thm:1st} with $b$ replaced by $b_{\varepsilon}$ and 
$C_{\varepsilon}$ is a positive constant depending on $\varepsilon$. 
By the continuity of $b_{\varepsilon}$ on $[0,T)\times\mathbb{R}^d$, 
\[
 \lim_{n\to\infty}b_{\varepsilon}(\tau_n(s),X_{\tau_n(s)}) = b_{\varepsilon}(s, X_s), \quad dt\times d\mathbb{P}^*\text{-a.e.}. 
\]
Thus, applying the bounded convergence theorem, we obtain 
\[
\lim_{n\to\infty}\mathbb{E}_{\mathbb{P}^*}\left[\int_0^T|\theta^{(n)}_s|^2ds\right]=0. 
\]
This together with $\lim_{\varepsilon\to 0}\delta(\mu_{\varepsilon},\mu)= 0$ completes the proof of the theorem. 
\end{proof}

%\begin{figure}[htbp]
%\centering
%\includegraphics[width=0.8\columnwidth, bb = 0 0 857 587]{fd_vs_col.png}
%\caption{The ratios of Max and RSM errors for $d=1$ with $n=2^8, 2^{12}$. }
%\label{fig:5.5}
%\end{figure}

%%%

\subsection*{Acknowledgements}

This study is supported by JSPS KAKENHI Grant Number JP21K03364.

%\nocite{*}
\bibliographystyle{plain}
\bibliography{../mybib}
%\bibliographystyle{junsrt}
%\bibliography{baibu}

\end{document}